\documentclass[a4paper,11pt]{article}
\usepackage{amsmath,amsthm,amsfonts,amssymb,enumerate,graphicx,float,wrapfig,calc,microtype,thmtools,underscore}

\usepackage[usenames,dvipsnames,svgnames,table]{xcolor}
\usepackage[unicode=true]{hyperref}
\hypersetup{
colorlinks,
linkcolor={black},
citecolor={black},
urlcolor={blue!60!black},
pdftitle={Minor-Closed Graph Classes with Bounded Layered Pathwidth},
pdfauthor={Vida Dujmovi{\'c}, David Eppstein, Gwena\"el  Joret, Pat Morin, David~R.~Wood}}
\usepackage[noabbrev,capitalise]{cleveref}

\usepackage[longnamesfirst,numbers,sort&compress]{natbib}
\makeatletter
\def\NAT@spacechar{~}
\makeatother

\usepackage[lmargin=29mm,rmargin=29mm,tmargin=29mm,bmargin=29mm]{geometry}
\renewcommand{\baselinestretch}{1.1}
\allowdisplaybreaks

\newcommand{\ceil}[1]{\ensuremath{\protect\lceil#1\rceil}}

\newcommand{\GG}{\ensuremath{\protect{\mathcal{G}}}}
\renewcommand{\ge}{\geqslant}
\renewcommand{\le}{\leqslant}
\renewcommand{\geq}{\geqslant}
\renewcommand{\leq}{\leqslant}

\DeclareMathOperator{\tw}{tw}
\DeclareMathOperator{\ltw}{ltw}
\DeclareMathOperator{\pw}{pw}
\DeclareMathOperator{\lpw}{lpw}
\DeclareMathOperator{\dist}{dist}

\newcommand{\subtree}[1]{\hat{#1}}

\renewcommand{\thefootnote}{\fnsymbol{footnote}}	
\allowdisplaybreaks

\newcommand{\arXiv}[1]{arXiv:\,\href{https://arxiv.org/abs/#1}{#1}}
\newcommand{\msn}[1]{MR:\,\href{https://www.ams.org/mathscinet-getitem?mr=MR#1}{#1}}

\newcommand{\doi}[1]{doi:\,\href{https://dx.doi.org/#1}{#1}}

\theoremstyle{plain}
\newtheorem{theorem}{Theorem}
\newtheorem{lemma}[theorem]{Lemma}
\newtheorem{corollary}[theorem]{Corollary}

\theoremstyle{definition}

\begin{document}

\author{
Vida Dujmovi{\'c}\,\footnotemark[3] 
\qquad David Eppstein\,\footnotemark[2] 
\qquad Gwena\"el  Joret\,\footnotemark[4] \\[1ex]
Pat Morin\,\footnotemark[1] 
\qquad David~R.~Wood\,\footnotemark[5]}

\footnotetext[3]{School of Computer Science and Electrical Engineering, University of Ottawa, Ottawa, Canada (\texttt{vida.dujmovic@uottawa.ca}). Research  supported by NSERC and the Ontario Ministry of Research and Innovation.}

\footnotetext[4]{D\'epartement d'Informatique, Universit\'e Libre de Bruxelles, Brussels, Belgium (\texttt{gjoret@ulb.ac.be}). Supported by an ARC grant from the Wallonia-Brussels Federation of Belgium.}

\footnotetext[2]{Department of Computer Science, University of California, Irvine, California, USA (\texttt{eppstein@uci.edu}). Supported in part by NSF grants  CCF-1618301 and CCF-1616248.}

\footnotetext[1]{School of Computer Science, Carleton University, Ottawa, Canada (\texttt{morin@scs.carleton.ca}). Research  supported by NSERC.}

\footnotetext[5]{School of Mathematics, Monash   University, Melbourne, Australia  (\texttt{david.wood@monash.edu}). Research supported by the Australian Research Council.}

\sloppy

\date{19th October 2018; revised 4th June 2020}

\title{\boldmath\bf Minor-Closed Graph Classes \\with Bounded Layered Pathwidth}
\maketitle


\begin{abstract}
We prove that a minor-closed class of graphs has bounded layered pathwidth if and only if some apex-forest is not in the class. This generalises a theorem of Robertson and Seymour, which says that a minor-closed class of graphs has bounded pathwidth if and only if some forest is not in the class. 
\end{abstract}

\renewcommand{\thefootnote}{\arabic{footnote}}

\section{Introduction}

Pathwidth and treewidth are graph parameters that respectively measure how similar a given graph is to a path or a tree. These parameters are of fundamental importance in structural graph theory, especially in Roberston and Seymour's graph minors series. They also have numerous applications in algorithmic graph theory. Indeed, many NP-complete problems are solvable in polynomial time on graphs of bounded treewidth \citep{FG06}. 

Recently, \citet{DMW17} introduced the notion of layered treewidth. Loosely speaking, a graph has bounded layered treewidth if it has a tree decomposition and a layering such that each bag of the tree decomposition contains a bounded number of vertices in each layer (defined formally below). This definition is interesting since several natural graph classes, such as planar graphs, that have unbounded treewidth have bounded layered treewidth. \citet{BDDEW19} introduced layered pathwidth, which is analogous to layered treewidth where the tree decomposition is required to be a path decomposition. 

The purpose of this paper is to characterise the minor-closed graph classes with bounded layered pathwidth.

\subsection{Definitions}

Before continuing, we define the above notions. A \emph{tree decomposition} of a graph $G$ is a collection $(B_x\subseteq V(G):x\in V(T))$ of subsets of $V(G)$ (called \emph{bags}) indexed by the nodes of a tree $T$, such that:
\begin{enumerate}[(i)]
\item for every edge $uv$ of $G$, some bag $B_x$ contains both $u$ and $v$, and 
\item for every vertex $v$ of $G$, the set $\{x\in V(T):v\in B_x\}$ induces a non-empty connected subtree of $T$.
\end{enumerate}
The \emph{width} of a tree decomposition is the size of the largest bag minus 1. The \emph{treewidth} of a graph $G$, denoted by $\tw(G)$, is the minimum width of a tree decomposition of $G$. 

A \emph{path decomposition} is a tree decomposition in which the underlying tree is a path. We denote a path decomposition by the corresponding sequence of bags $(B_1,\dots,B_n)$. The \emph{pathwidth} of $G$, denoted by $\pw(G)$, is the minimum width of a path decomposition of $G$. 

A graph $H$ is a \emph{minor} of a graph $G$ if a graph isomorphic to $H$ can be obtained from a subgraph of $G$ by contracting edges. A class of graphs \GG\ is \emph{minor-closed} if for every $G\in \GG$, every minor of $G$ is in \GG. 

A \emph{layering} of a graph $G$ is a partition $(V_0,V_1,\dots,V_t)$ of $V(G)$ such that for every edge $vw\in E(G)$, if $v\in V_i$ and $w\in V_j$ then $|i-j|\leq 1$. Each set $V_i$ is called a \emph{layer}.  For example, for a vertex $r$ of a connected graph $G$, if $V_i$ is the set of vertices at distance $i$ from $r$, then $(V_0,V_1,\dots)$ is a layering of $G$, called the \emph{bfs layering} of $G$ starting from $r$. 

\citet{DMW17} introduced the following definition. The \emph{layered width} of a tree decomposition $(B_x:x\in V(T))$ of a graph $G$ is the minimum integer $\ell$ such that, for some layering $(V_0,V_1,\dots,V_t)$ of $G$, each bag $B_x$ contains at most $\ell$ vertices in each layer $V_i$. The \emph{layered treewidth} of a graph $G$, denoted by $\ltw(G)$, is the minimum layered width of a tree decomposition of $G$. \citet{BDDEW19} defined the \emph{layered pathwidth} of a graph $G$, denoted by $\lpw(G)$, to be the minimum layered width of a path decomposition of $G$.

For a graph parameter $f$ (such as treewidth, pathwidth, layered treewidth or layered pathwidth), a graph class $\mathcal{G}$ has \emph{bounded} $f$ is there exists a constant $c$ such that $f(G)\leq c$ for every graph $G\in\mathcal{G}$. 

\subsection{Examples and Applications}

Several interesting graph classes have bounded layered treewidth (despite having unbounded treewidth). For example, \citet{DMW17} proved that every planar graph has layered treewidth at most 3, and more generally that every graph with Euler genus $g$ has layered treewidth at most $2g+3$. Note that layered treewidth and layered pathwidth are not minor-closed parameters (unlike treewidth and pathwidth). In fact, several graph classes that contain arbitrarily large clique minors have bounded layered treewidth or bounded layered pathwidth. For example, \citet{DEW17} proved that every graph that can be drawn on a surface of Euler genus $g$ with at most $k$ crossings per edge has layered treewidth at most $2(2g+3)(k+1)$. Even with $g=0$ and $k=1$, this family includes graphs with arbitrarily large clique minors.  Map graphs have similar behaviour~\citep{DEW17}.

\citet{BDDEW19} identified the following natural graph classes\footnote{A \emph{squaregraph} is a graph that has an embedding in the plane in which each bounded face is a 4-cycle and each vertex either belongs to the unbounded face or has degree at least 4. A \emph{Halin graph} is a planar graph obtained from a tree $T$ with at least four vertices and with no vertices of degree 2 by adding a cycle through the leaves of $T$ in the clockwise order determined by a plane embedding of $T$. For a set $P$ of points in the plane, the \emph{unit disc graph} $G$ of $P$ has vertex set $P$, where $vw \in E(G)$ if and only if $\dist(v, w) \leq 1$.} that have bounded layered pathwidth (despite having unbounded pathwidth): every squaregraph has layered pathwidth 1; every bipartite outerplanar graph has layered pathwidth 1; every outerplanar graph has layered pathwidth at most 2; every Halin graph has layered pathwidth at most 2; and every unit disc graph with clique number $k$ has layered pathwidth at most $4k$. 

Part of the motivation for studying graphs with bounded layered treewidth or pathwidth is that such graphs have several desirable properties. For example, Norin proved that every $n$-vertex graph with layered treewidth $k$ has treewidth less than $2\sqrt{kn}$ (see \citep{DMW17}). This leads to a very simple proof of the Lipton-Tarjan separator theorem. A standard trick leads to an upper bound of $11\sqrt{kn}$ on the pathwidth (see \citep{DEW17}).

Another application is to stack layouts (or book embeddings), queue layouts and track layouts. \citet{DMW17} proved that every $n$-vertex graph  with layered treewidth $k$ has track- and queue-number $O(k\log n)$. This leads to the best known bounds on the track- and queue-number of several natural graph classes\footnote{Subsequent to the submission of this paper, improved bounds on the queue-number have been obtained~\citep{DJMMUW}. Still it is open whether graphs of bounded layered treewidth have bounded queue-number.}. For graphs with bounded layered pathwidth, the dependence on $n$ can be eliminated: \citet{BDDEW19} proved that every graph with layered pathwidth $k$ has track- and queue-number at most $3k$. Similarly, \citet{DMY} proved that every graph with layered pathwidth $k$ has stack-number at most $4k$. 

Graph colouring is another application area for layered treewidth. \citet{EJ14} proved that every graph with maximum degree $\Delta$ and Euler genus $g$ is (improperly) 3-colourable with bounded clustering, which means that each monochromatic component has size bounded by some function of $\Delta$ and $g$. This resolved an old open problem even in the planar case $(g=0)$. The clustering function proved by \citet{EJ14} is roughly $O(\Delta^{32\Delta\,2^g})$. While \citet{EJ14} made no effort to reduce this function, their method will not lead to a  sub-exponential clustering bound. On the other hand, \citet{LW} proved that every graph with layered treewidth $k$ and maximum degree $\Delta$ is 3-colourable with clustering $O(k^{19}\Delta^{37})$, which was recently improved to $O(k^3\Delta^2)$ by \citet{DEMWW}. In particular, every graph with Euler genus $g$ and maximum degree $\Delta$ is 3-colorable with clustering $O(g^3\Delta^2)$. This greatly improves upon the clustering bound of \citet{EJ14}. Moreover, the proofs in \citep{LW,DEMWW} are relatively simple, avoiding many technicalities that arise when dealing with graph embeddings. These results highlight the utility of layered treewidth as a general tool. 

\subsection{Characterisations}

We now turn to the question of characterising those minor-closed classes that have bounded treewidth. The key example is the $n\times n$ grid graph, which has treewidth $n$. Indeed, \citet{RS-V} proved that every graph with sufficiently large treewidth contains the $n\times n$ grid as a minor. The next theorem follows since every planar graph is a minor of some grid graph. Several subsequent works have improved the bounds \citep{RST94,LeafSeymour15,DJGT-JCTB99,Chuzhoy15,CC16,CT19}. 

\begin{theorem}[\citet{RS-V}]
\label{TreewidthCharacterisation}
A minor-closed class has bounded treewidth if and only if some planar graph is not in the class. 
\end{theorem}

An analogous result for pathwidth holds, where the complete binary tree is the key example (the analogue of grid graphs for treewidth). Let $T_h$ be the complete binary tree of height $h$ (the rooted tree in which each non-leaf vertex has exactly two children, and the distance from the root to each leaf vertex equals $h$). It is well known and easily proved that  $\pw(T_h)=\ceil{\frac{h}{2}}$, and every forest is a minor of some complete binary tree. \citet{RS-I} proved the following characterisation. 

\begin{theorem}[\citet{RS-I}]
\label{PathwidthCharacterisation}
A minor-closed class has bounded pathwidth if and only if some forest is not in the class. 
\end{theorem}

Note that \citet{BRST-JCTB91} proved the following quantitatively stronger result: for every forest $T$ with $|V(T)|\geq 2$ every graph containing no $T$ minor has pathwidth at most $|V(T)|-2$. 

Now consider layered analogues of \cref{TreewidthCharacterisation,PathwidthCharacterisation}. A graph $G$ is \emph{apex} if $G-v$ is planar for some vertex $v$. Define the $n\times n$ \emph{pyramid} to be the apex graph obtained from the $n\times n$ grid by adding one dominant vertex $v$. (Here a vertex is \emph{dominant} if it is adjacent to every other vertex in the graph.)\ The $n\times n$ pyramid has treewidth $n+1$ and layered treewidth at least $\frac{n+2}{3}$, since every layering uses at most three layers. Pyramids are `universal' apex graphs, in the sense that every apex graph is a minor of some pyramid graph (since every planar graph is a minor of some grid graph). \citet{DMW17} proved the following characterisation. 

\begin{theorem}[\citet{DMW17}]
\label{LayeredTreewidthCharacterisation}
A minor-closed class has bounded layered treewidth if and only if some apex graph is not in the class. 
\end{theorem}

\cref{LayeredTreewidthCharacterisation} generalises the above-mentioned result that graphs of bounded Euler genus have bounded layered treewidth. 

A graph $G$ is an \emph{apex-forest} if $G-v$ is a forest for some vertex $v$. 
The following analogue of \cref{LayeredTreewidthCharacterisation} is the main result of this paper. 

\begin{theorem}
\label{LayeredPathwidthCharacterisation}
A minor-closed class has bounded layered pathwidth if and only if some apex-forest is not in the class. 
\end{theorem}


\citet{DMW17} noted that \cref{LayeredTreewidthCharacterisation} implies \cref{TreewidthCharacterisation}. Similarly, we now show\footnote{While \cref{LayeredPathwidthCharacterisation} implies \cref{PathwidthCharacterisation}, we emphasise that our proof of \cref{LayeredPathwidthCharacterisation} relies on the above-mentioned quantitative version of \cref{PathwidthCharacterisation} by \citet{BRST-JCTB91}. Similarly, the proof of 
\cref{LayeredTreewidthCharacterisation} uses the graph minor structure theorem of \citet{RS-XVI} and thus relies on 
\cref{TreewidthCharacterisation}.} that \cref{LayeredPathwidthCharacterisation} implies \cref{PathwidthCharacterisation}. Let $T$ be a forest, and let $G$ be a graph with no $T$ minor. Let $T^+$ be the apex-forest obtained from $T$ by adding a dominant vertex $v$. Let $G^+$ be the graph obtained from $G$ by adding a dominant vertex $x$. Suppose for the sake of contradiction that $G^+$ contains a $T^+$-minor.  A $T^+$-minor in $G^+$ can be described by a mapping from the vertices of $T^+$ to vertex-disjoint trees in $G^+$ such that whenever two vertices in $T^+$ are adjacent, the corresponding two trees induce a connected subgraph of $G$. From this mapping, remove two (not necessarily distinct) trees, the image of $v$ and the tree (if it exists) that contains $x$. If the tree that contains $x$ was the image of a vertex $w$ in $T$, then instead map $w$ to the tree that was the image of $v$. The resulting mapping describes a $T$-minor in $G$, as claimed. This contradiction shows that $G^+$ is $T^+$-minor-free. By \cref{LayeredPathwidthCharacterisation}, $G^+$ has layered pathwidth at most $c=c(T^+)$. Since $G^+$ has radius 1, at most three layers are used. Thus $G^+$ and $G$ have pathwidth less than $3c$.


Layered treewidth is closely related to the notion of `local treewidth', which was first introduced by \citet{Eppstein-Algo00} under the guise of the `treewidth-diameter' property. A graph class $\mathcal{G}$ has \emph{bounded local treewidth} if there is a function $f$ such that for every graph $G$ in $\mathcal{G}$, for every vertex $v$ of $G$ and for every integer $r\geq0$, the subgraph of $G$ induced by the vertices at distance at most $r$ from $v$ has treewidth at most $f(r)$. If $f(r)$ is a linear function, then  $\mathcal{G}$ has \emph{linear local treewidth}. See \citep{Grohe-Comb03,DH-SJDM04,DH-SODA04,Eppstein-Algo00,FG06} for results and algorithmic applications of local treewidth.  \citet{DMW17} observed that if some class $\mathcal{G}$ has bounded layered treewidth, then $\mathcal{G}$ has linear local treewidth. On the other hand, bounded layered treewidth is a stronger property than bounded or linear local treewidth.

Local pathwidth is defined similarly to local treewidth. A graph class $\mathcal{G}$ has \emph{bounded local pathwidth} if there is a function $f$ such that for every graph $G$ in $\mathcal{G}$, for every vertex $v$ of $G$ and for every integer $r\geq0$, the subgraph of $G$ induced by the vertices at distance at most $r$ from $v$ has pathwidth at most $f(r)$. The observation of \citet{DMW17} extends to the setting of local pathwidth; see \cref{lpw-llp} below. 

\cref{LayeredPathwidthCharacterisation} is extended to capture local pathwidth by the following theorem, which also provides a structural description in terms of a tree decomposition with certain properties that we now introduce. If $T$ is a tree indexing a tree decomposition of a graph $G$, then for each vertex $v$ of $G$, let $T[v]$ denote the subtree of $T$ induced by those nodes corresponding to bags that contain $v$. Thus $T[v]$ is non-empty and connected. Say that a tree decomposition of a graph $G$ is \emph{$(w,p)$-good} if its width is at most $w$ and, for every $v\in V(G)$, the subtree $T[v]$ has pathwidth at most $p$. We illustrate this definition with two examples. Let $T$ be a tree, rooted at some vertex. For each node $x$ of $T$, introduce a bag $B_x$ consisting of $x$ and its parent node (or just $x$ if $x$ is the root). Then $(B_x:x\in V(T))$ is a tree decomposition of $T$ with width 1. Moreover, for each vertex $v$, the subtree $T[v]$ is a star, which has pathwidth 1. Thus every tree has a $(1,1)$-good tree decomposition. Now, consider an outerplanar triangulation $G$. Let $T$ be the weak dual tree (ignoring the outerface). For each node $x$ of $T$, let $B_x$ be the set of three vertices on the face corresponding to $x$. Then $(B_x:x\in V(G))$ is a tree decomposition of $G$ with width 2. Moreover, for each vertex $v$ of $G$, the subtree $T[v]$ is a path, which has pathwidth 1.  Thus every outerplanar graph has a $(2,1)$-good tree decomposition (since every outerplanar graph is a subgraph of an outerplanar triangulation). These constructions are generalised via the following theorem, which immediately implies \cref{LayeredPathwidthCharacterisation}. 

\begin{theorem}
\label{MainThm}
The following are equivalent for a minor-closed class \GG:
\begin{enumerate}[(1)]
\item some apex-forest graph is not in \GG, 
\item \GG\ has bounded local pathwidth, 
\item \GG\ has linear local pathwidth. 
\item \GG\ has bounded layered pathwidth, 
\item there exist integers $w$ and $p$, such that every graph in \GG\ has a $(w,p)$-good tree decomposition. 
\end{enumerate}
\end{theorem}

Here is some intuition about property (5). Suppose that \GG\ excludes some apex-forest graph as a minor. Since every apex-forest graph is planar, by \cref{TreewidthCharacterisation}, the graphs in \GG\ have bounded treewidth. Thus we should expect that the tree decompositions in (5) have bounded width. Moreover, if \GG\ has bounded layered pathwidth, then $G[N(v)]$ has bounded pathwidth for each vertex $v$ in each graph $G\in\GG$. Property (5) takes this idea further, and says that each subtree $T[v]$ has bounded pathwidth, which implies that $G[N(v)]$ has bounded pathwidth (since the width of the tree decomposition is bounded; see \cref{BlowUp}). 

In \cref{easy-section} we prove $(5)\implies (4)\implies(3)\implies(2)\implies(1)$. In \cref{hard-section} we close the loop by proving $(1)\implies (5)$. This proof uses a recent characterisation by \citet{Dang18}  of the unavoidable minors in 3-connected graphs of large pathwidth. 

Throughout the proof we use the following `universal' apex-forest graph. Let $Q_k$ be the graph obtained from the complete binary tree $T_k$ by adding one dominant vertex. Note that $\pw(Q_k)=\ceil{\frac{k}{2}}+1$ and the layered pathwidth of $Q_k$ is at least $\frac{k+4}{6}$, since every layering of $Q_k$ uses at most three layers. Since every forest is a minor of some complete binary tree, every apex-forest graph is a minor of some $Q_k$. 

\section{Downward Implications}
\label{easy-section}

We start with a few simple but useful lemmas.

\begin{lemma}
\label{BlowUp}
If a graph $G$ has a tree decomposition of width $k$ indexed by a tree of pathwidth $p$, then $G$ has pathwidth at most $(p+1)(k+1)-1$.
\end{lemma}

\begin{proof}
Let $(B_x:x\in V(T))$ be a tree decomposition of $G$ of width $k$. Let $(C_1,\dots,C_n)$ be a path decomposition of $T$ of width $p$. For $i\in\{1,\dots,n\}$, let $D_i:= \bigcup_{x\in C_i}B_x$. Then $(D_1,\dots,D_n)$ is a path decomposition of $G$ of width $(p+1)(k+1)-1$ (since $|C_i|\leq p+1$ and $|B_x|\leq k+1$). 
\end{proof}

\begin{lemma}
Let $T_1$ and $T_2$ be subtrees of a tree $T$, such that $T=T_1\cup T_2$. Then 
$$\pw(T)+1\leq(\pw(T_1)+1)+(\pw(T_2)+1).$$
\end{lemma}

\begin{proof}
Let $(B_1,\dots,B_s)$ be a path decomposition of $T_1$ with bag size at most $\pw(T_1)+1$. Each component  of $T-V(T_1)$ is contained in $T_2$ and therefore has a path decomposition  with bag size at most $\pw(T_2)+1$. For each such component $J$ of $T-V(T_1)$, there is exactly one vertex $v$ in $T_1$ adjacent to some vertex in $J$ (otherwise $T$ would contain a cycle consisting of two edges between a path in $T_1$ and a path in $J$). Say $v$ is in bag $B_i$. We say $J$ \emph{attaches} at $v$ and at $B_i$. By doubling bags in the path decomposition of $T_1$, we may assume that distinct components of $T-V(T_1)$ attach at distinct $B_i$. For each component $J$ of $T-V(T_1)$, if $(D_1,\dots,D_t)$ is a path decomposition of $J$ with bag size at most $\pw(T_2)+1$, then replace $B_i$ by $(B_i\cup D_1,\dots,B_i\cup D_t)$. We obtain a path decomposition of $T$ with bag size at most $(\pw(T_1)+1)+(\pw(T_2)+1)$. The result follows. 
\end{proof}

\begin{corollary}
\label{SubtreeUnion}
Let $T_1,\dots,T_k$ be subtrees of a tree $T$, such that $T=T_1\cup\dots\cup T_k$. Then 
\begin{equation*}
\pw(T)+1\leq\sum_{i=1}^k(\pw(T_i)+1).
\end{equation*} 
\end{corollary}

We now prove the downward implications in \cref{MainThm}. First note that (2) implies (1), since if every graph in \GG\ has local pathwidth at most $k$, then the apex-forest graph $Q_{6k}$ is not in \GG. It is immediate that (3) implies (2). That (4) implies (3) is the above-mentioned observation of \citet{DMW17} specialised for pathwidth. We include the proof for completeness. 

\begin{lemma}
\label{lpw-llp}
Let \GG\ be a class of graphs such that every graph in \GG\ has layered pathwidth at most $k$. Then \GG\ has linear local pathwidth with binding function $f(r) =(2r+1)k-1$.  
\end{lemma}

\begin{proof}
For a graph $G\in\GG$, let $(B_1,\dots,B_s)$ be a path decomposition of $G$ with layered width $k$, with respect to some layering $(V_0,V_1,\dots,V_t)$. Let $v$ be a vertex in $V_i$. Let $r$ be a positive integer. Let $H$ be the subgraph of $G$ induced by the vertices at distance at most $r$ from $v$. Thus $V(H)\subseteq V_{i-r}\cup V_{i-r+1}\cup\dots\cup V_{i+r}$. Each bag $B_j$ contains at most $k$ vertices in each layer. Hence $(B_1\cap V(H),\dots,B_s\cap V(H))$ is a path decomposition of $H$ with at most $(2r+1)k$ vertices in each bag. Therefore \GG\ has linear local pathwidth with binding function $f(r) =(2r+1)k-1$.  
\end{proof}

The next lemma shows that (5) implies (4). 

\begin{lemma}
If a graph $G$ has a $(w,p)$-good tree decomposition, then 
$$\lpw(G)\leq w(p+1)(w+1).$$
\end{lemma}

\begin{proof} Let  $\mathcal{T}=(B_x:x\in V(T))$ be a tree decomposition of $G$ with width $w$, such that $\pw(T[v])\leq p$  for each vertex $v$ of $G$. Since adding edges does not decrease the layered pathwidth, we may add edges to $G$ between two non-adjacent vertices in the same bag of $\mathcal{T}$. Now each bag is a clique, and $G$ is chordal with maximum clique size $w+1$. Let $(V_0,V_1,\dots,V_t)$ be a bfs layering in $G$. That is, $V_i$ is the set of vertices in $G$ at distance $i$ from some fixed vertex $r$ of $G$. In particular, $V_0=\{r\}$.

Consider a component $H$ of $G[V_i]$ for some $i\geq 1$. Let $C_H$ be the set of vertices in $V_{i-1}$ adjacent to at least one vertex in $H$. Since $G$ is chordal, $C_H$ is a clique of size at most $w$ (see \citep{KP-DM08,DMW05}), called the \emph{parent clique} of $H$. Define $T_H:=\bigcup_{u\in C_H}T[u]$. Since $C_H$ is a clique, which is contained in a single bag of $\mathcal{T}$, there is a node $x$ of $T$ such that $x\in T[u]$ for each $u\in C_H$. Thus $T_H$ is a (connected) subtree of $T$. Moreover, $T_H$  is the union of at most $w$ subtrees, each with pathwidth at most $p$. Thus $\pw(T_H)+1\leq w(p+1)$ by \cref{SubtreeUnion}. Let $\hat{H}:= G[V(H)\cup C_H]$. 

We now prove that $\mathcal{T}_H := (B_x\cap V(\hat{H}) : x \in V(T_H))$ is a tree decomposition of $\hat{H}$. We first prove condition (ii). For a vertex $v$ of $C_H$, the set of bags of $\mathcal{T}_H$ that contain $v$ is precisely those indexed by nodes in $T[v]$, which is non-empty and connected, by assumption. Now, consider a vertex $v$ in $H$. Let $w$ be the neighbour of $v$ on a shortest $vr$-path in $G$. Thus $w$ is in $C_H$. Since $vw$ is an edge, $v$ and $w$ appear in a common bag of $\mathcal{T}$, which corresponds to a node in $T_H$ (since that bag contains $w$). Hence $T_H[v]$ is non-empty. 
We now prove that  $T_H[v]$ is connected. Let $B_1$ and $B_2$ be distinct bags of $\mathcal{T}_H$ containing $v$. Let $P$ be the $B_1B_2$-path in $T$. Since $T[v]$ is connected, $v$ is in the bag associated with each node in $P$. To conclude that $T_H[v]$ is connected, it remains to prove that $P\subseteq T_H$. By construction, some vertex $w_1$ is in $B_1\cap C_H$ and some vertex $w_2$ is in $B_2\cap C_H$. Since $w_1$ and $w_2$ are adjacent, the bag associated with each node in $P$ contains $w_1$ or $w_2$. Hence $P\subseteq T_H$ and $T_H[v]$ is connected. This proves condition (ii). Now we prove condition (i). Since $C_H$ is contained in some bag of $\mathcal{T}_H$, condition (i) holds for each edge with endpoints in $C_H$. For each edge $vw$ with $v\in V(H)$ and $w\in C_H$,  $v$ and $w$ are in a common bag $B_x$ of $\mathcal{T}$, implying $x$ is in $T_H$ (since $B_x$ contains $w$), as desired. Finally, consider an edge $uv$ with $u,v\in V(H)$. Suppose on the contrary that $u$ and $v$ have no common neighbour in $C_H$. By construction,  $u$ has a neighbour $w_1$ in $C_H$, and $v$ has a neighbour $w_2$ in $C_H$. Thus $w_1\neq w_2$. Since $C_H$ is a clique, $w_1$ and $w_2$ are adjacent. Since $uw_2\not\in E(G)$ and $vw_1\not\in E(G)$, the 4-cycle $(u,w_1,w_2,v)$ is chordless, and $G$ is not chordal, which is a contradiction. Hence $u$ and $v$ have a common neighbour $w$ in $C_H$. Thus $\{u,v,w\}$ is a triangle in $G$, which is in a common bag of $T$, and therefore in a common bag of $\mathcal{T}_H$, implying that $u$ and $v$ are in a common bag of $\mathcal{T}_H$. This proves condition (i) in the definition of tree decomposition. Therefore $\mathcal{T}_H$ is a tree decomposition of $\hat{H}$. By construction, it has width at most $w$. 

Since $\pw(T_H)+1\leq w(p+1)$ and $T_H$ indexes a tree decomposition of $\hat{H}$ with width at most $w$, by \cref{BlowUp}, $\pw(\hat{H}) \leq (w(p+1)+1)(w+1)-1$.  

We now construct a path decomposition of $G$ with layered width at most $w(p+1)(w+1)$ with respect to layering $(V_0,V_1,\dots,V_t)$. Let $G_i:=G[V_0\cup V_1\cup\dots\cup V_i]$. We now prove, by induction on $i$, that $G_i$ has a path decomposition with layered width at most $w(p+1)(w+1)$ with respect to layering $(V_0,V_1,\dots,V_i)$. This claim is trivial for $i=0$. Now assume that $(B_1,\dots,B_q)$ is a path decomposition of $G_{i-1}$ with layered width at most $w(p+1)(w+1)$ with respect to layering $(V_0,V_1,\dots,V_{i-1})$. For each component $H$ of $G[V_i]$, there is a bag $B_j$ that contains $C_H$; pick one such bag and call it the \emph{parent bag} of $H$. By doubling the bags, we may assume that distinct components of $G[V_i]$ have distinct parent bags. Now, for each component $H$ of $G[V_i]$ with parent bag $B_j$, if $(D_1,\dots,D_s)$ is a path decomposition of $\hat{H}$ with width  $w(p+1)(w+1)-1$, then replace $B_j$ by $(B_j\cup D_1,\dots,B_j\cup D_s)$. Doing this for each component of $G[V_i]$ produces a path decomposition of $G_i$  with layered width at most $w(p+1)(w+1)$ with respect to layering $(V_0,V_1,\dots,V_i)$. In particular, we obtain a path decomposition of $G$  with layered width at most $w(p+1)(w+1)$ with respect to layering $(V_0,V_1,\dots,V_t)$.
 \end{proof}

\section{Proof that (1) implies (5)}
\label{hard-section}

The goal of this section is to show that if a graph $G$ excludes some apex-forest graph $H$ as a minor, then $G$ has a $(w,p)$-good tree decomposition for some $w=w(H)$ and $p=p(H)$. Since every apex-forest graph is a minor of some $Q_k$, it suffices to prove this result for $H=Q_k$, in which case we denote $w=w(k)$ and $p=p(k)$.



We will be working with two related trees $S$ and $T$ and one graph $G$. To help the reader keep track of things we use variables $a$, $b$, and $c$ as names for \emph{nodes} of $S$ and $T$ and variables $v$, $x$, $y$, and $z$ to refer to \emph{vertices} of $G$.

We now give an outline of the proof.  First, we show that a recent result by  \citet{Dang18}  implies that every 3-connected graph $G$ with no $Q_k$ minor has pathwidth at most $w=w(k)$.  Thus, in this case, $G$ has a $(w,1)$-good tree decomposition.  Next we deal with cut vertices by showing that if each block of a graph $G$ has a $(w,p)$-good tree decomposition, then $G$ has a $(w,p+1)$-good tree decomposition.

Therefore, the main difficulty is to show that every 2-connected graph $G$ with no $Q_k$ minor has a $(w,p)$-good tree decomposition $(B_a:a\in V(T))$.  By the result of \citet{BRST-JCTB91} described in the introduction, if $\pw(T[v])> 2^{h+1}-3$ for some $v\in V(G)$ then $T[v]$ contains a $T_h$ minor.  For sufficiently large $h$, we then construct a $Q_k$ minor (from the $T_h$ minor in $T[v]$) in which $v$ plays the role of the apex vertex.

To construct the tree decomposition $(B_a:a\in V(T))$ we use two tools:
An SPQR-tree, $S$, represents a graph $G$ as a collection of subgraphs (S- and
R-nodes) that are joined at 2-vertex cutsets (P-nodes).  These subgraphs
consist of cycles (S-nodes) and 3-connected graphs (R-nodes). Cycles have
pathwidth 2 and, by the result of Dang discussed above, the 3-connected
graphs have pathwidth at most $w=w(k)$. Replacing the S- and R-nodes of the
SPQR-tree with these path decompositions produces the tree $T$ in our
tree decomposition.

To show that this tree decomposition is $(w,p)$-good, we first show that
if $T[v]$ contains a subdivision of a sufficiently large complete binary
tree, then the SPQR-tree $S$ also contains a subdivision of a large
complete binary tree all of whose nodes have subgraphs that contain $v$.
Using this large binary tree in $S$ we then piece together a subgraph of $G$
that has a $Q_k$ minor in which $v$ is the apex vertex.

\subsection{Dang's Result}

First we show how the following result of \citet{Dang18} implies that every 3-connected graph with no $Q_k$ minor has pathwidth at most $w=w(k)$.

\begin{theorem}[\mbox{\citet[Theorem 1.1.5]{Dang18}}]\label{dang-theorem}
  Let $P$ be a graph with two distinct vertices $u_1$ and $u_2$ such
  that $P-\{u_1,u_2\}$ is a forest, $Q$ be a graph with a vertex $v$
  such that $Q-v$ is outerplanar, and $R$ be a
  tree with a cycle going through its leaves in order from the leftmost
  leaf to the rightmost leaf so that $R$ is planar.  Then there exists
  a number $w = w(P, Q , R)$ such that every 3-connected graph of
  pathwidth at least $w$ has a $P$, $Q$, or $R$ minor.
\end{theorem} 

Note that $R$ is a Halin graph, except that degree-2 vertices are allowed in the tree. 


To use \cref{dang-theorem} we need a small helper lemma.  For every $k\ge 0$, let $T_k^+$ be the graph obtained from the complete binary tree $T_k$ of height $k$ by adding a new vertex adjacent to the leaves. The next lemma is well known. 

\begin{lemma} 
\label{FindQh}
For every integer $k\ge 0$, $T_{2k}^+$ contains $Q_k$ as a minor.
\end{lemma}

\begin{proof}
  The statement is immediate for $k=0$. For $k\ge 1$, partition the
  edges of $T_{2k}^+$ into the tree $T_{2k}$ and the remaining edges,
  which form a star centered at some vertex $v$.  Let $a_1,a_2,a_3,a_4$
  be the grandchildren of the root of $T_{2k}$ ordered from left to right.
  Contract the entire subtree comprised of the subtree rooted at $a_2$,
  the subtree rooted at $a_3$, and the path from $a_2$ to $a_3$. Applying
  the same procedure recursively on the copy of $T_{2(k-1)}^+$ rooted at $a_1$ and
  the copy of $T_{2(k-1)}^+$ rooted at $a_4$ produces $Q_k$, as can be easily verified
  by induction.
\end{proof}

\begin{corollary}\label{3-connected}
  There exists a number $w=w(k)$ such that every 3-connected graph of
  pathwidth at least $w$ has a $Q_k$ minor.
\end{corollary}

\begin{proof}
Let $P$ be obtained from the complete binary tree $T_k$ by adding two dominant vertices $u_1$ and $u_2$. Let  $Q$ be the graph obtained from the outerplanar graph $\nabla_{k}$, whose weak dual is a complete binary tree of height $k$, by adding a dominant vertex $v$.  Let $R$ be the graph obtained from $T_{2k +1}$ by adding a cycle on its leaves, so that $R$ is planar. 

Then $P$ contains $Q_k$ as a minor since $P-\{u_2\}$ is isomorphic
  to $Q_k$.  $Q$ also contains $Q_k$ as a minor because $\nabla_{k}$
  contains a complete binary tree of height $k$ as a subgraph. Finally,
  $R$ also contains a $Q_k$ minor: Contract the cycle, then we have
  a complete binary tree of height $2k$ plus an apex vertex linked
  to its leaves, which contains $Q_k$ as a minor by \cref{FindQh}.
  \cref{dang-theorem} implies that there exists $w=w(k)$ such that every
  3-connected graph with pathwidth at least $w$ contains at least one
  of $P$, $Q$, or $R$ as a minor and therefore contains a $Q_k$ minor.
\end{proof}

\subsection{Dealing with Cut Vertices}

A \emph{block} in a graph is either a maximal 2-connected subgraph, the subgraph  induced by the endpoints of a bridge edge, or the subgraph induced by an isolated vertex. 

\begin{lemma}\label{cut-vertices}
Let $G$ be a graph, such that each block of $G$ has a $(w,p)$-good tree decomposition. 
Then $G$ has a $(w,p+1)$-good tree decomposition.
\end{lemma}

 
\begin{proof}
  Let $C_1,\ldots,C_r$ be the blocks of $G$. For each $i\in\{1,\dots,r\}$, let $T_i$ be the underlying tree in a $(w,p)$-good tree decompositions of $C_i$. 

  We create a tree decomposition of $G$ as follows: For each cut vertex or isolated vertex 
  $v$ in $G$,  introduce a new tree node $a_v$ with $B_{a_v}=\{v\}$.
  In each block $C_i$ that contains $v$, the tree decomposition $(B_a:a\in V(T_i))$ of $C_i$ has at least one node $a$ such that $v\in B_a$; make $a_v$ adjacent to exactly one such node for each $C_i$.  

  It is straightforward to verify that this defines a tree decomposition of $G$
  and we now argue this decomposition is $(w,p+1)$-good.
  The resulting tree decomposition of $G$
  has width at most $w$.  
  For each isolated vertex $v\in V(G)$, the subtree $T[v]$ consists of one node.   
  For each cut vertex $v\in V(G)$, the subtree $T[v]$ 
  is   composed of some number of subtrees, each adjacent to $a_v$ and each
  having a path decomposition of width at most $p$.  We obtain a
  path decomposition of $T[v]$ by concatenating the path decompositions
  of each subtree and adding $v$ to every bag of the resulting path
  decomposition. The resulting path decomposition of $T[v]$ has width
  at most $p+1$.
\end{proof}

\subsection{SPQR-Trees}

In this section, we quickly review SPQR-trees, a structural
decomposition of 2-connected graphs used previously to characterize
planar embeddings \cite{MacLane37}, to design efficient
algorithms for triconnected components
\cite{HT73}, and in efficient data structures for
incremental planarity testing \cite{dibattista.tamassia:on-line, dibattista.tamassia:on-line2}.



Let $G$ be a 2-connected graph. An \emph{SPQR-tree} $S$ of $G$ is a tree in which each node $a\in V(S)$ is associated with a minor $H_a$ of $G$.  For any S- or R-node $a$ of $S$, $H_a$ is a simple graph. If $a$ is a P-node, on the other hand, then $H_a$ is a \emph{dipole graph} having two vertices and at least two parallel edges.  In all cases, $H_a$ is a minor of $G$.  For a P-node $a$ in which $H_a$ contains vertices $x$ and $y$ and $t$ parallel edges, this means that $G$ contains $t$ internally disjoint paths from $x$ to $y$. For each node $a$ of $S$ each edge $xy\in E(H_a)$ is classified either as a \emph{virtual edge} or a \emph{real edge}. An SPQR-tree $S$ is defined recursively as follows (see \cref{spqr}):\footnote{This definition includes P-nodes consisting of only two virtual edges, which some works exclude because they are unnecessary. However, their inclusion simplifies some of our analysis.}

\begin{figure}
  \begin{center}
    \includegraphics[width=\textwidth]{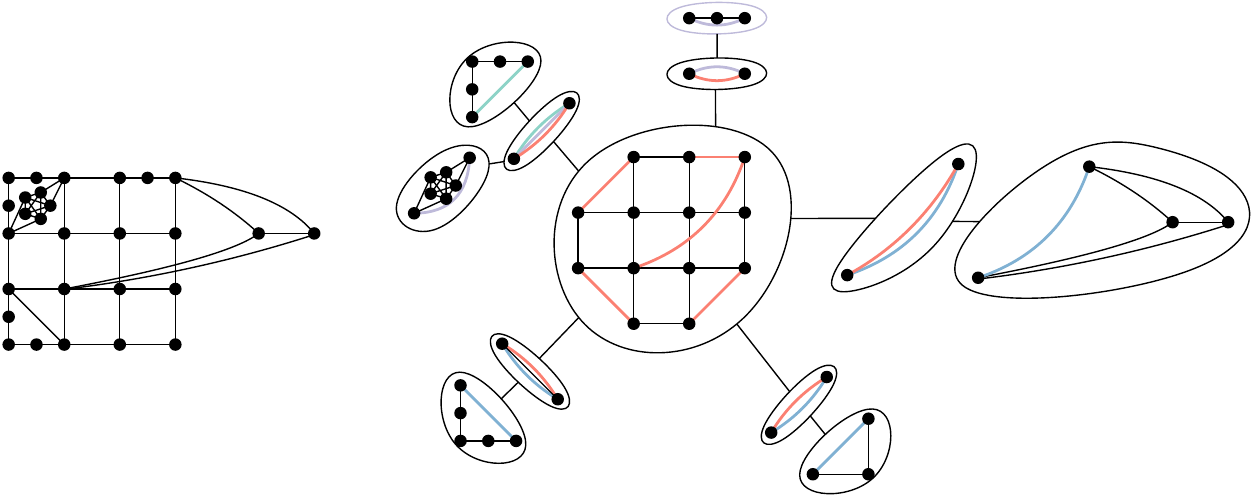}
  \end{center}
  \caption{A graph and its SPQR-tree.}
  \label{spqr}
\end{figure}
\begin{enumerate}
   \item If $G$ is a cycle, then $S$ consists of a single node $a$
     (an S-node) in which $H_a=G$ and all edges of $H_a$ are real.

   \item If $G$ is 3-connected, then $S$ consists of a single node $a$
     (an R-node) in which $H_a=G$ and all edges of $H_a$ are real.

   \item    Otherwise $G$ has a cutset $\{x,y\}$ such that 
     $x$ and $y$ each have degree at least 3.
     Then let $C_1,\ldots,C_r$, $r\ge 2$, be the connected components of
     $G-\{x,y\}$.  For each $i\in \{1,\ldots,r\}$, let $\tilde{G}_i$
     be $G[V(C_i)\cup\{x,y\}]$ along with the additional edge $xy$, if not
     already present. (The edge $xy$ is treated as a real edge so that the node $a_i$ mentioned below  exists.)\ Because of the inclusion of $xy$, each $\tilde{G}_i$  is
     2-connected, so each has an SPQR-tree $S_i$.  Then an SPQR-tree
     for $G$ is obtained by creating a node $a$ (a P-node) with $H_a$
     being a dipole graph with vertices $x$ and $y$ and having $r$
     virtual edges joining $x$ and $y$. In addition to these virtual edges, 
     $H_a$ contains the real edge $xy$ if $xy\in E(G)$.  The construction and the fact that $xy$ is an edge in each $\tilde{G}_i$ imply that, for each $i\in\{1,\ldots,r\}$,  there exists exactly one node $a_i$ in $S_i$ such that $xy$ is a real edge in $H_{a_i}$. To complete $S$, make $a$ adjacent to each of $a_1,\ldots,a_r$, and make $xy$ a virtual edge in each of $H_{a_1},\ldots,H_{a_r}$.

\end{enumerate}

Let $S$ be an SPQR-tree of a 2-connected graph $G$. For each node $a$ of $S$, let $E_r(H_a)$ denote the set of real edges in $H_a$ and let $E_v(H_a)$ denote the multiset of virtual edges in $H_a$.  For a connected subtree $S'$ of $S$, define $G[S']$ to be the subgraph of $G$ whose vertex set is $V(G[S'])=\bigcup_{a\in V(S')} V(H_a)$ and whose edge set is $E(G[S'])=\bigcup_{a\in V(S')} E_r(H_a)$. For a vertex $v\in V(G)$, let $S[v]:=S[\{a\in V(S): v\in V(H_a)\}]$, which is called the subtree of $S$ \emph{induced by} $v$. We make use of the following properties of $S$:

\begin{enumerate}
   \item Every R-node and S-node is adjacent only to P-nodes and no two
   P-nodes are adjacent.
   \item The degree of every node $a$ is equal to the number of virtual edges
     in $H_a$.
   \item For every vertex $v\in V(G)$, $S[v]$ is connected.
   \item If $a$ is an R-node or S-node, then $H_a$ is a simple graph; 
   that is, $H_a$ contains no parallel edges.
   \item If a P-node $a$ has degree 2 and both its neighbors are S-nodes 
      then $H_a$ has a real edge.
    \item For each node $a$ of $S$ and component $S'$ of $S-\{a\}$,
     $G[S']$ is connected.
   \item For each $xy\in E(G)$ there is exactly one node $a$ of $S$ 
    for which $xy$ is a real edge in $H_a$.
\end{enumerate}

\subsection{The Good Tree Decomposition}
\label{good-tree}

To obtain our good tree decomposition $(B_a:a\in V(T))$ of a 2-connected graph $G$ we start with an SPQR-tree $S$ for $G$.  For each R-node or S-node $a$ of $S$, let $(B_c: c\in V(P_a))$ be a minimum-width path decomposition of $H_a$.  
The tree $T$ includes the path $P_a$ for each R-node and S-node $a$ of $S$. We say that the node $a$ of $S$ \emph{generates} the nodes in the path $P_a$ and that each node in $P_a$ is \emph{generated by} $a$. Each S- or R-node $a$ is adjacent to some set of P-nodes in $S$. For each such P-node $b$ whose dipole graph $H_b$ has vertices $x$ and $y$, the edge $xy$ is a (virtual) edge in $H_a$ and therefore $x$ and $y$ appear in some common bag $B_c$ with $c\in V(P_a)$. Add $c$ and $b$ as vertices in $T$ and add the edge $bc$ to $T$.  Add $b$ as a node of $T$ and add the edge $bc$ to $T$. This defines the tree $T$ in the tree decomposition.

We now describe the contents of $T$'s bags.  
Each P-node $a$ of $S$ becomes a node in $T$ whose bag contains
only the two vertices of $H_a$.  Every node $a$ in $T$ that is generated
by an S- or R-node $a'$ of $S$ is a node in some path decomposition of
$H_{a'}$ and already has an associated bag $B_{a}$
that it inherits from this path decomposition.

It is straightforward to verify that $(B_a:a\in V(T))$ is indeed a tree decomposition of $G$: For each vertex $v\in V(G)$, the connectivity of the subtree $T[v]$ follows from Property~3 of SPQR-trees and the equivalent property for the path decompositions that include $v$. Each edge $xy$ of $G$ appears as an edge in $H_a$ for at least one node $a$ of $S$ and therefore $x$ and $y$ appear in a common bag in the path decomposition of $H_a$.

Each bag $B_a$ of $(B_a:a\in V(T))$ either has size in $\{2,3\}$
(when $a$ is generated by a P-node or an S-node) or it has size at
most $w(k)+1$ where $w(k)$ is the function in \cref{3-connected} (when
$a$ is generated by an R-node).  Thus, $(B_a:a\in V(T))$ is a tree
decomposition of $G$ whose width is upper bounded by a function of $k$.
It remains to show that, for every $v\in V(G)$, $T[v]$ has pathwidth
that is upper bounded by a function of $k$.

In the remainder of this section, we fix $G$ to be a 2-connected graph,
$S$ to be an SPQR-tree of $G$, and $(B_a:a\in V(T))$ to be a tree
decomposition of $G$ obtained using the procedure described above.  

\begin{lemma}
  \label{roof-lifting}
  For every integer $h\ge 1$, if $T[v]$ has pathwidth greater than $2^{2h+1}-3$, then $S[v]$ contains
  a subdivision of $T_h$.
\end{lemma}

\begin{proof}
  In the following, a \emph{binary tree} is a tree rooted at a degree-2 node, such that every other node has degree in $\{1,2,3\}$. In a binary tree, the root and every degree 3 node is called a \emph{branching node}.  Every branching node and every leaf is a \emph{distinctive node}.  We use the convention that all binary  trees are ordered, possibly arbitrarily, so that we can distinguish between the left and right child of a branching node.  For a node $a$  in a binary tree $T$, we denote by $\subtree{a}$ the subtree rooted at $a$; that is, the subtree of $T$ induced by the set of nodes that have $a$ as an ancestor, including $a$ itself.

  Recall, from the result of \citet{BRST-JCTB91} discussed in the introduction,
  that if $T[v]$ has pathwidth greater than $2^{2h+1}-3$ then $T$
  contains a subdivision $T'$ of $T_{2h}$.  Note that $T'$ does not
  immediately imply the existence of $T_h$ in $S[v]$ since two or more
  distinctive nodes of $T'$ may have been generated by the same node of $S$.  
  Label each node of $T'$ with the node of $S$
  that generated it.  Recall that each node $a$ in $S$ generates a path
  in $T$. So a maximal subset of nodes of $T'$ with a common
  label induces a path in $T'$.

  We claim that $T'$ contains a subdivision $T''$ of $T_{h}$ such that no two distinctive nodes of $T''$ have the same label.  We establish this claim by induction on $h$: If $h=0$ then the claim is trivial.
  Otherwise, let $a$ be the root of $T'$ and let $a'$ and
  $a''$ be the highest branching nodes in the left and right subtrees
  of $\subtree{a}$, respectively.  Let $a_1$ and $a_2$ be the highest
  distinctive nodes in the left and right subtrees of $\subtree{a}'$,
  respectively, and let $a_3$ and $a_4$ be the highest distinctive nodes in
  the left and right subtrees of $\subtree{a}''$, respectively.  Since each
  label induces a path in $T'$, at least
  one of $\{a_1,a_2\}$, say $a_1$, and at least one of $\{a_3,a_4\}$, say $a_4$, does not have the same label as $a$.
  Furthermore, since $a_1$ and $a_4$ are separated by $a$, the set of labels
  of nodes in $\subtree{a}_1$ is disjoint from the set of labels of
  nodes in $\subtree{a}_4$.  Applying induction on $\subtree{a}_1$
  and $\subtree{a}_4$ yields two subdivisions of $T_{h-1}$ in which no two distinctive nodes have the same label.  Connecting these two subdivisions with the unique path from $a_1$ to $a_4$ yields the
  desired subdivision of $T_{h}$ in which no two distinctive nodes have the same label.

  Since no two distinctive nodes in $T''$ have the same label, 
  each distinctive node corresponds to a unique node of $S$.  Thus, contracting all nodes of
  $T''$ that share a common label yields a subtree $T'''$ of $S[v]$ that is a subdivision of $T_{h}$.
\end{proof}

Thus far we have established that if $T[v]$ has sufficiently high
pathwidth, then $S[v]$ contains a subdivision of a large complete binary
tree.

\begin{lemma}\label{qk-minor}
  If $S[v]$ contains a subdivision of $T_{7(k+1)}$ then $G$ contains a 
  $Q_{k}$ minor.
\end{lemma}

\begin{proof}
  First we note that if $S[v]$ contains a subdivision of $T_{7(k+1)}$
  then $S[v]$ contains a subdivision $T'$ of $T_{k+1}$ such that the path
  between each pair of distinctive nodes in $T'$ has length at least 7.

  It is convenient to work with a simplified SPQR-tree $S'$ and graph $G'$
  obtained by repeating the following operation exhaustively: Consider
  some edge $ab$ of $S$ with $a\in V(T')$ and $b\not\in V(T')$. The edge
  $ab$ is associated with some virtual edge $xy$ in $H_a$.  In $S'$, 
  replace the virtual edge $xy$ in $H_a$ with a real edge. At the same
  time, remove the maximal subtree $\subtree{b}$ of $S$ that contains
  $b$ and not $a$. By Property~6 of SPQR trees, in $G'$ this operation
  is equivalent to contracting all the real edges in $\bigcup_{c\in
  V(\subtree{b})} E_r(H_c)$ and removing any resulting parallel edges.
  Since the resulting graph $G'$ is a minor of $G$, this operation is
  safe in the sense that the existence of a $Q_k$ minor in $G'$ implies
  the existence of a $Q_k$ minor in $G$.

  With this simplification, the tree $T'$ is an SPQR-tree for the graph $G'$ and every virtual edge is incident to $v$.  We now turn our efforts to finding the $Q_k$ minor in $G'$.  Recall that $Q_k$ is obtained from a complete binary tree $T_k$ by adding a dominant vertex. We begin by finding a subdivision $T''$ of $T_{k+1}$ in $G'$. In this subdivision, each edge of $T_{k+1}$ that joins a node to its left child is represented by a path $P_{\mu\nu}$ joining a branching node $\mu$ to a distinctive node $\nu$. We show that $G'$ contains a path from $v$ to some \emph{anchor node} $\eta$ of $P_{\mu\nu}$ with $\eta\neq\nu$, which is vertex disjoint from $T''$ except for $\eta$.  Furthermore, except for their common endpoint $v$, all of these paths are disjoint. The union of $T''$ and these paths contains a $Q_k$ minor since contracting the path from each anchor node to its closest ancestor branching node produces $Q_k$. See \cref{qk}.

  \begin{figure}
    \begin{center}
      \includegraphics{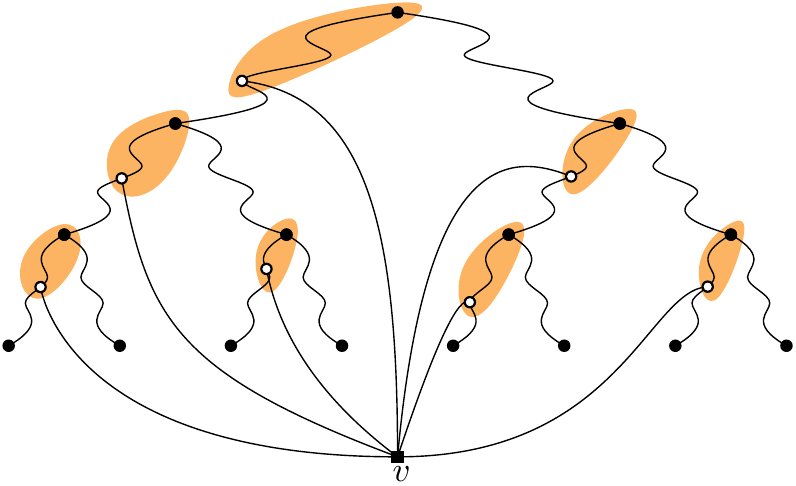}
    \end{center}
    \caption{Finding a $Q_k$ minor in the proof of \cref{qk-minor}. Distinctive nodes are indicated by black disks and anchor nodes by white circles.}
    \label{qk}
  \end{figure}

  Let $a$ be a branching node of $T'$ and let $b$ be the nearest
  distinctive node in one of $a$'s two subtrees.  Consider the path
  $a=c_1,c_2,\ldots,c_r=b$ in $T'$.  For each $i\in\{1,\ldots,r-1\}$,
  the edge $c_ic_{i+1}$ is associated with a cutset $\{v,x_i\}$ in
  $G'$ and $vx_i$ is a virtual edge in $H_{c_i}$ and $H_{c_{i+1}}$.
  Note that this implies that, for each $i\in\{2,\ldots,r\}$, $H_{c_i}$
  contains both vertices $x_i$ and $x_{i-1}$.

  We claim that, for each $i\in\{2,\ldots,r-1\}$,
  $H_{c_{i}}$ contains a path $P_i$ from $x_{i-1}$ to $x_{i}$ that
  does not contain $v$; refer to \cref{paths}.  When $c_i$ is a P-node, this claim is trivial
  since, in this case, $x_{i-1}=x_{i}$.  The case in which $c_i$
  is an S-node or R-node is also easy: In these cases $H_{c_i}$ is
  2-connected, therefore there is a path from $x_{i-1}$ to $x_{i}$
  that avoids $v$.  Now note that the paths $P_1,\ldots,P_{r-1}$ are
  disjoint, except for each of the common endpoints $x_i$ where $P_{i}$ ends and $P_{i+1}$ begins.  This is because each $\{v,x_i\}$ is 
  a cutset of $G'$ that
  separates $\bigcup_{j=1}^{i} V(H_{c_j})\setminus\{v,x_i\}$ from
  $\bigcup_{j={i+1}}^r V(H_{c_j})\setminus\{v,x_i\}$.  By concatenating
  $P_2,\ldots,P_{r-1}$ we obtain a path $P_{ab}$ from $x_1\in V(H_a)$
  to $x_{r-1}\in V(H_b)$ that we call the \emph{subdivision path} for
  nodes $a$ and $b$.

  \begin{figure}[!t]
    \begin{center}
      \includegraphics{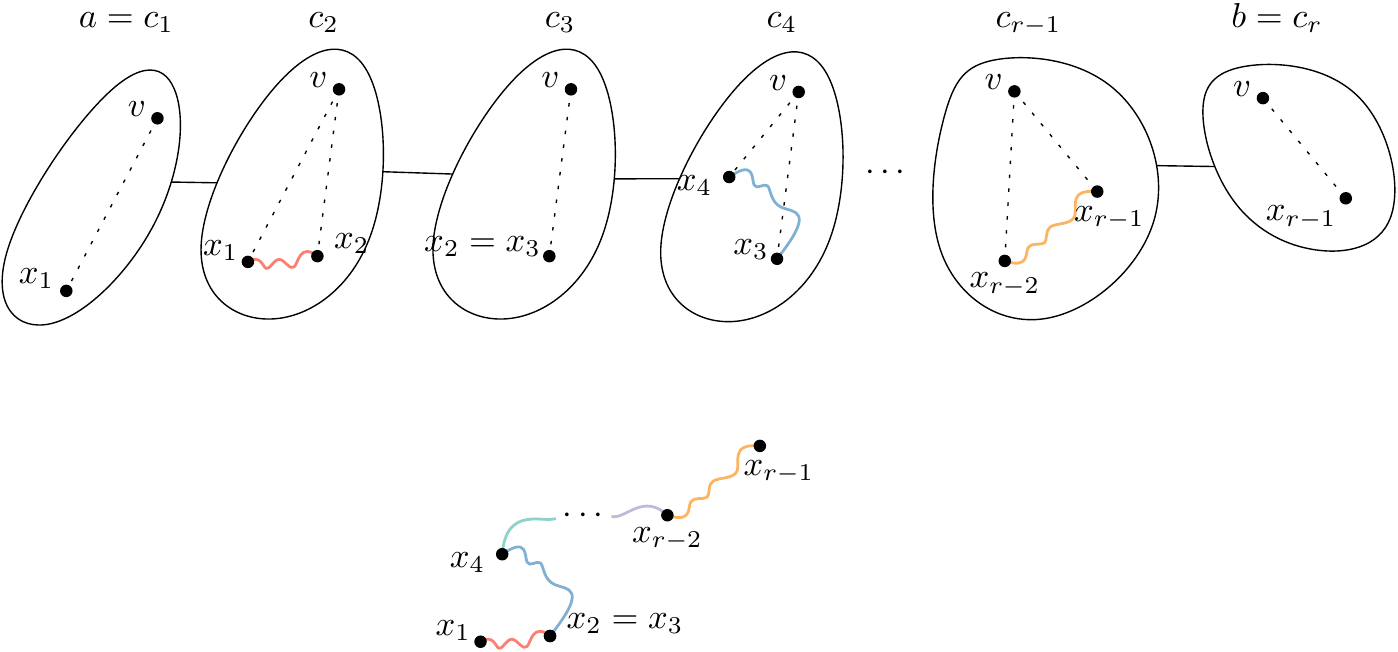}
    \end{center}
    \caption{Finding a path connecting a vertex of $H_a$ to a vertex of $H_b$.}
    \label{paths}
  \end{figure}


Consider any branching node $a$ of $T'$.  If $a$ is a P-node then all
the subdivision paths that begin or end at a vertex of $H_a$ include
the same vertex of $H_a$. If $a$ is an S- or R-node, each subdivision
path that begins or ends at a vertex of $H_a$ includes a different
vertex (for up to 3 different vertices $x$, $y$, and
$z$).\footnote{The only time $a$ may be an S-node is when $a$ is the
root of $T'$ in which case the two subdivision paths that begin at a
vertex in $H_a$ begin at the two neighbours $x$ and $y$ of $v$ in the
cycle $H_a$.} Now, since $H_a$ is 2-connected, these three vertices
are in the same component of $H_a - \{v\}$. In particular, $H_a -
\{v\}$ contains an edge-minimal tree that includes $x$, $y$, and $z$.
Adding each of these trees to the union of all subdivision paths
produces the subdivision $T''$ of $T_{k+1}$.


  Next we show how to construct paths from $v$ to anchor nodes.
  Let $a$ be a branching node of $T'$, let $b$ be the highest
  distinctive node in $a$'s left subtree and let $c_1=a,\ldots,c_r=b$
  be the path in $T'$ having endpoints $a$ and $b$.   
  Thus far we have established
  that $G'[\{c_2,\ldots,c_{r-1}\}]$ contains a simple path $P_{ab}$
  from $x_1$ to $x_{r-1}$ that does not include $v$.   Note that  $r$ is large enough so that $c_5$ exists.  We now show
  that $G'[\{c_2,c_3,c_4,c_5\}]$ contains an \emph{apex path} $P'$
  from $v$ to some \emph{anchor node} of $P_{ab}$ such that the
  internal vertices of $P'$ are disjoint from $V(P_{ab})\cup V(H_{b})$.
  We first describe the path $P'$ in $G'[\{c_2,c_3,c_4,c_5\}]$ and then
  show that $P'$ contains no vertex of $V(H_b)\setminus\{v\}$.
  There are two cases to consider:
  \begin{enumerate}
    \item $c_i$ is an R-node, for some $i\in\{2,\ldots,5\}$:  Since
    $H_{c_i}$ is 3-connected, there are three paths in $H_{c_i}$
    with endpoints $v$ and $x_i$ and no other vertices in common.
    Since $H_{c_{i}}$ has only two virtual edges, at least one of these
    paths uses only real edges in $H_{c_i}$.  This path therefore
    contains a subpath $P'$ joining $v$ to some vertex of $P_{ab}$
    (the anchor vertex) that is otherwise disjoint from $P_{ab}$.

%

\item Otherwise, none of $c_2,\ldots,c_5$ is an R-node.  Property~1 of
SPQR-trees implies that, for at least one $i\in\{2,3\}$, $c_i$ is an
S-node, $c_{i+1}$ is a P-node, and $c_{i+2}$ is an S-node.  In the
simplified SPQR-tree $S'$, $c_2$ and $c_3$ each have degree 2.
Therefore, Property~5 of SPQR-trees (applied to $S'$) ensures that
$H_{c_{i+1}}$ contains the real edge $vx_{i+1}$. This real edge in the
version of $H_{c_{i+1}}$ that appears in $S'$ corresponds either to a
real edge in the version of $H_{c_{i+1}}$ that appears in $S$ or it
was introduced in $S'$.  In the former case, $G$ contains the edge
$vx_{i+1}$ and we are done.  In the latter case $c_{i+1}$ is adjacent
in $S$ to a node $d$ that is not in $T'$ and the maximal subtree
$\hat{d}$ of $S$ that contains $d$ but not $c_{i+1}$ was removed from
$S$ while producing $S'$. Recall that this is equivalent to
contracting all the real edges $\bigcup_{c\in V(\hat{d})} E_r(H_c)$.
Thus, the real edge in the version of $H_{c_{i+1}}$ that appears in
$S'$ implies the existence of a path in $G$ from $v$ to $x_{i+1}$
whose internal vertices appear only in $\bigcup_{c\in V(\hat{d})}
V(H_c)$. The internal vertices of this path are disjoint from
$P_{ab}$.

  \end{enumerate}
  It remains to show that $P'$ does not contain any vertices of
  $V(H_b)\setminus \{v\}$.  By Properties~1 and 6 of SPQR-trees,
  for each $x\in V(G')\setminus\{v\}$, the subtree $T'[x]$ of $T'$
  consisting only of nodes $a$ such that $x\in V(H_a)$ is a star; that is,
  the distance between any two nodes in $T'[x]$ is at most 2.  Now,
  since the distance between any two distinctive nodes of $T'$ is at least
  $7$, we have $r\ge 8$ and therefore $H_b=H_{c_r}$ has no vertex, except $v$,
  in common with any of $H_{c_2},\ldots,H_{c_5}$.  Therefore, $P'$
  joins $v$ to a vertex in $P_{ab}$ that is not in $H_{b}$, as required.

  Adding the set of all apex paths to $T''$ then produces a subgraph of
  $G'$ that contains a $Q_k$-minor.
\end{proof}




Finally, we have all the pieces in place to complete the proof.

\begin{proof}[Proof that (1) implies (5)]
Let $G$ be a graph excluding some apex-forest graph $H$ as a minor. As explained earlier, $G$ contains no $Q_k$ minor for some $k=k(H)$.  We wish to show that there are $w$ and $p$ that depend only on $k$ such that $G$ has a $(w,p)$-good tree decomposition.  By \cref{cut-vertices} we may assume that $G$ is 2-connected.

Consider the good tree decomposition $(B_a:a\in V(T))$ of $G$ described in   \cref{good-tree}.    This decomposition has width at most $w$ where   $w=w(k)$ is the function that appears (implicitly) in \cref{dang-theorem} and \cref{3-connected}.   We claim that, for each $v\in V(G)$, $\pw(T[v])\le   2^{14(k+1)+1}-3$, so that this tree decomposition is   $(w,2^{14(k+1)+1}-3)$-good.  Otherwise, by \cref{roof-lifting}, there is a vertex $v\in V(G)$ such that an SPQR-tree $S$ has a subtree $S[v]$ that contains a subdivision of $T_{7(k+1)}$.  Therefore, by \cref{qk-minor}, $G$ contains a $Q_k$ minor, contradicting the supposition that $G$ has no $Q_k$ minor.
\end{proof}


Note that we have not tried to optimise constants in the above proof. For example, with more work the constant 14 can be reduced to less than 3. 



Finally, we address the computational complexity of our main theorem. 

\begin{theorem}
   There exists a function $f:\mathbb{N}\to\mathbb{N}$ and an algorithm
   that takes as input an $n$-vertex graph and outputs, in $O(f(k)n)$
   time, an $(f(k),f(k))$-good tree decomposition of $G$ and a layered
   path decomposition of $G$ of layered width at most $f(k)$, where $k$
   is largest integer such that $G$ contains a $Q_k$ minor.
\end{theorem}

\begin{proof}[Proof Sketch]
   In the following, for each $i\in\{0,\ldots,5\}$, $f_i:\mathbb{N}\to\mathbb{N}$ is an unspecified function that is known to exist.  Since $G$ has no $Q_k$-minor, $|E(G)|\leq f_0(k)n$ (see \citep{ReedWood16}).  An SPQR-tree $S$ of $G$ can be computed in $O(f_0(k)n)$ time, and the total size of the graphs $\{H_a:a\in V(S)\}$ is at most $O(f_0(k)n)$ (see \cite{GM00}).  A path decomposition of $H_a$ with width at most 2 for each S-node or P-node $a$ is easily computed in time linear in the size of $H_a$.  For each R-node $a$, the pathwidth of $H_a$ is at most $p=p(k)$. Taking $p(k)$ to be part of $f_1(k)$, a minimum-width path decomposition of $H_a$ for an R-node $a$ can be computed in $O(f_1(k)|V(H_a)|)$ time \cite{BK91,K93,CDF96,Bodlaender-SJC96}. These path decompositions are all that is needed to construct the tree $T$ and an $(f_2(k),f_3(k))$-good tree decomposition $\{ B_a:a\in V(T)\}$ in $O(f_1(k)n)$ time.



   The proof that (5) implies (4) in \cref{easy-section} is constructive
   and immediately gives an $O(f_4(k)n)$ time algorithm to convert
   the $(f_2(k),f_3(k))$-good tree decomposition into a layered path
   decomposition of layered width at most $f_5(k)$. This establishes the
   theorem for $f(k) = \max\{f_i(k) : i\in\{0,\ldots,5\}\}$.
\end{proof}

\subsection*{Acknowledgements} Thanks to the anonymous referees for pointing out some missing details and for several other helpful comments. Some of this research took place at the 2018 Graphs@IMPA workshop in Rio de Janeiro, February 2018.

  \let\oldthebibliography=\thebibliography
  \let\endoldthebibliography=\endthebibliography
  \renewenvironment{thebibliography}[1]{%
    \begin{oldthebibliography}{#1}%
      \setlength{\parskip}{0.3ex}%
      \setlength{\itemsep}{0.3ex}%
  }{\end{oldthebibliography}}

\def\soft#1{\leavevmode\setbox0=\hbox{h}\dimen7=\ht0\advance \dimen7
  by-1ex\relax\if t#1\relax\rlap{\raise.6\dimen7
  \hbox{\kern.3ex\char'47}}#1\relax\else\if T#1\relax
  \rlap{\raise.5\dimen7\hbox{\kern1.3ex\char'47}}#1\relax \else\if
  d#1\relax\rlap{\raise.5\dimen7\hbox{\kern.9ex \char'47}}#1\relax\else\if
  D#1\relax\rlap{\raise.5\dimen7 \hbox{\kern1.4ex\char'47}}#1\relax\else\if
  l#1\relax \rlap{\raise.5\dimen7\hbox{\kern.4ex\char'47}}#1\relax \else\if
  L#1\relax\rlap{\raise.5\dimen7\hbox{\kern.7ex
  \char'47}}#1\relax\else\message{accent \string\soft \space #1 not
  defined!}#1\relax\fi\fi\fi\fi\fi\fi}

\end{document}